\documentclass[11pt,letterpaper]{amsart}
\usepackage{amssymb,amsmath,amsthm}
\usepackage{mathrsfs}
\usepackage{graphicx}
\usepackage[colorlinks=true,citecolor=blue,linkcolor=red,urlcolor=black]{hyperref}
\usepackage{xcolor}

\newtheorem{theorem}{Theorem}[section]
\newtheorem{lemma}[theorem]{Lemma}

\newtheorem{proposition}[theorem]{Proposition}

\theoremstyle{remark}
\newtheorem{remark}{Remark}

\numberwithin{equation}{section}

\def\i{\mathrm{i}}
\newcommand{\C}{\mathbb C}
\newcommand{\R}{\mathbb R}
\newcommand{\D}{\mathbb D}
\newcommand{\dd}{\,d}
\newcommand{\dA}{\,dA}
\newcommand{\supp}{\operatorname{supp}}

\newcommand{\bpartial}{\overline\partial}
\newcommand{\pt}{\widetilde p}
\newcommand{\bmu}{\boldsymbol\mu}
\newcommand{\bpsi}{\boldsymbol\psi}

\newcommand{\loc}{\mathrm{loc}}
\DeclareMathOperator{\diag}{diag}
\allowdisplaybreaks[1]

\title[The planar Calder\'on problem]{The planar Calder\'on problem for $L^p$ potentials}
\author[Ali Feizmohammadi]{Ali Feizmohammadi}
\address{Department of Mathematics, University of Toronto, Bahen Centre, 40 St. George St., Toronto, Ontario, Canada, M5S 2E4}
\email{ali.feizmohammadi@utoronto.ca}
\begin{document}

	\begin{abstract}
		We prove that the Dirichlet-to-Neumann map associated to the Schr\"odinger operator $-\Delta+q$ on a planar bounded Lipschitz domain uniquely determines an arbitrary complex-valued potential $q\in L^p$, $p>1$, whenever zero is not a Dirichlet eigenvalue. Our argument uses Nachman's d-bar equation only at sufficiently large complex frequencies, where exceptional points are absent. A new annular frequency estimate converts the first asymptotic coefficient of the complex geometric optics solutions into a pointwise bound for a primitive of the potential difference, reducing the inverse problem to a unique-continuation argument.
	\end{abstract}

	\maketitle
	
	\section{Introduction}
	
	Let $\Omega$ be a planar bounded Lipschitz domain, and let $q\in L^p(\Omega;\C)$, for some $p>1$. We consider the boundary value problem
	\begin{equation}\label{bvp}
		\begin{cases}
			(-\Delta+q)u=0 & \text{in }\Omega,\\
			u=f & \text{on }\partial\Omega.
		\end{cases}
	\end{equation}
	Throughout the paper, all function spaces are over $\C$. We assume that zero is not a Dirichlet eigenvalue of $-\Delta+q$. Under this assumption, it is classical that given any $f\in H^{1/2}(\partial\Omega)$ there is a unique weak solution $u\in H^1(\Omega)$ to the boundary value problem \eqref{bvp}. The associated Dirichlet-to-Neumann (DN) map
	$$
	\Lambda_q:H^{1/2}(\partial\Omega)\longrightarrow H^{-1/2}(\partial\Omega)
	$$
	is defined via the complex bilinear pairing
	\begin{equation}\label{dn_map}
		\langle\Lambda_q f,g\rangle
		=\int_\Omega\big(\nabla u\cdot\nabla w+q u w\big)\dd x,
	\end{equation}
	where $u$ solves \eqref{bvp} and $w\in H^1(\Omega)$ is any function that satisfies $w|_{\partial\Omega}=g$. 
	
	The inverse problem is to determine $q$ given the knowledge of the boundary map $\Lambda_q$. We prove the following uniqueness result.
	
	\begin{theorem}\label{main_thm}
		Let $\Omega$ be a planar bounded domain with Lipschitz boundary. For $j=1,2$, let $q_j\in L^p(\Omega;\C)$ for some $p>1$ and assume that zero is not a Dirichlet eigenvalue for $-\Delta+q_j$ on $\Omega$. Then,
		$$
		\Lambda_{q_1}=\Lambda_{q_2} \qquad \implies \qquad q_1=q_2.
		$$
	\end{theorem}
	
	The study of these questions originates in Calder\'on's seminal work \cite{Calderon1980}, where the inverse conductivity problem was formulated and uniqueness was established for its linearization at a constant conductivity. In dimensions $n\geq3$, Sylvester and Uhlmann \cite{SU1987}
	established global uniqueness for the Schr\"odinger inverse boundary
	value problem with smooth potentials by constructing complex
	geometric optics (CGO) solutions. The extension to bounded potentials
	was obtained in
	\cite[Theorem~1.6]{NSU88}.
	In dimension two, Nachman \cite{Nachman1996} proved uniqueness and
	reconstruction for compactly supported real-valued potentials
	$q\in L^p(\R^2)$, $p>1$, under the additional assumption that
	\begin{equation}\label{q_cond}
		q=\frac{\Delta\sqrt{\gamma}}{\sqrt\gamma}
		\quad\text{in }\R^2,
		\qquad
		\gamma\in W^{2,p}_{\loc}(\R^2),
		\qquad
		0<c_0\leq\gamma\leq C_0.
	\end{equation}
	Under the above assumptions the Schr\"odinger equation admits a globally bounded,
	uniformly positive solution. This assumption connects
	the problem to the inverse conductivity problem: the substitution $v=\sqrt{\gamma} u$ transforms
	the conductivity equation $\nabla\cdot(\gamma\nabla u)=0$ into $(-\Delta+q)v=0$. The conductivity structure plays an essential role in Nachman's global d-bar method: it excludes exceptional points and
	provides the control near zero complex frequency needed to solve
	the d-bar problem on the whole complex plane
	\cite{Nachman1996}. 
	
	Let us also remark that the distinction between the planar and higher-dimensional inverse problem above
	is already apparent from a formal count of variables. The
	Schwartz kernel of the DN map depends on $2(n-1)$ variables,
	whereas the unknown potential depends on $n$ variables. The problem
	is therefore formally overdetermined for $n\geq3$ and formally
	determined for $n=2$. As emphasized in \cite{Nachman1996}, this
	loss of overdeterminacy is a fundamental difficulty in dimension two
	and is reflected in the different uses of CGO solutions in the two settings. In particular, while the large complex-frequency asymptotics of CGO solutions sufficed in dimensions $n\geq 3$, Nachman's argument relied on information at all complex-frequencies leading to the conductivity type assumption \eqref{q_cond} on the potential.
	
	\subsection{Previous literature}
	
	Following Nachman's global uniqueness
	theorem \cite{Nachman1996}, Brown and Uhlmann \cite{BU97}
	established uniqueness for uniformly positive conductivities
	in $W^{1,p}$, $p>2$. Astala and P\"aiv\"arinta \cite{AP06}
	subsequently resolved the planar conductivity problem
	for bounded measurable conductivities.
	For unbounded conductivities on simply connected domains,
	C\^arstea and Wang \cite{CW18} proved uniqueness in $W^{1,2}$,
	assuming a positive lower bound and smallness of
	$\|\nabla\log\gamma\|_{L^2}$.
	Nachman, Regev, and Tataru \cite{NRT20} established a nonlinear
	Plancherel theorem and used it to reconstruct positive
	conductivities satisfying $\log\gamma\in\dot H^1$, with the
	boundary normalization $\gamma=1$ on $\partial\Omega$.
	Their result requires no smallness assumption and permits
	both unboundedness and degeneracy.
	
	For the planar Schr\"odinger equation, Sun and Uhlmann
	\cite{SU91} established local uniqueness near generic potentials
	and global uniqueness for generic pairs of potentials.
	Isakov and Nachman \cite[Theorem~1.3]{IN95} proved uniqueness
	for real-valued $L^p$ potentials, $p>1$, with well-defined
	DN maps, assuming in addition that the first
	Dirichlet eigenvalue of one of the two operators is strictly
	positive.
	Lakshtanov and Vainberg \cite{LV17} developed a reconstruction
	procedure for certain compactly supported real-valued $L^p$ potentials,
	$p>1$, allowing exceptional complex frequencies.
	For each prescribed spatial point, their reconstruction equation
	is uniquely solvable nearby for an open dense set of potentials;
	the generic set may depend on the point, and the neighborhood
	may depend on the potential \cite[Theorem~2.3]{LV17}.
	
	All of the aforementioned results are based on CGO solutions with linear phases. A different approach was introduced by Bukhgeim \cite{Bukhgeim2008}, who constructed CGO
	solutions with quadratic phases and proved uniqueness for
	continuously differentiable potentials. The critical point of the phase can be prescribed at an
	arbitrary interior point. Substitution of suitable pairs of
	these solutions into the well-known Alessandrini identity then
	identifies the potential at that point by stationary phase.
	Thus, the argument uses only large values of the parameter
	and does not require solving a global d-bar problem
	in the complex-frequency plane.
	
	At lower regularity, Bl\aa sten, Imanuvilov, and Yamamoto \cite{BIY15}
	proved uniqueness for $L^p$ potentials with $p>2$.
	Bl\aa sten, Tzou, and Wang \cite{BTW20} extended uniqueness
	to complex-valued potentials in $L^p$, $p>4/3$.
	Their construction of CGO solutions
	is valid for every $p>1$, while their uniqueness argument
	requires $p>4/3$. For a further account of Nachman's d-bar method and Bukhgeim's CGO method with phases that admit critical points, including their use on Riemannian surfaces, we refer to the survey \cite{GT13}. 
	
	Let us also briefly recall the connection with inverse scattering at a fixed positive energy.
	The relation between scattering data and exterior Green
	functions goes back to Berezanskii \cite{Berezanskii1958}.
	The passage from these Green functions to DN maps is
	developed in \cite[Theorem~1.6]{Nachman1988} and
	\cite[Section~2]{SU93}; see also \cite{GKLU09}.
	Novikov \cite{Novikov1992} developed a reconstruction framework
	for the two-dimensional fixed-energy problem using generalized
	scattering data.
	Sun and Uhlmann \cite{SU93} showed that fixed-energy scattering
	data determine the interfaces and jumps of compactly supported
	piecewise smooth potentials.
	For unbounded potentials, Astala, Faraco, and Rogers
	\cite{AFR16} obtained almost everywhere reconstruction for
	compactly supported $H^{1/2}$ potentials from boundary
	measurements and fixed-energy scattering data. 
	
	Before closing this section, let us mention that while completing this manuscript, the author became aware of independent work of C. I. C\^arstea and J.-N. Wang, which proves the same theorem by substantially different methods. The two works were developed independently, and the respective authors agreed to post their manuscripts simultaneously on arXiv.
	
	\subsection{Outline of the main ideas}
	
	The proof connects the asymptotic behaviour of CGO solutions with unique continuation in the spatial variable.
	Its main ingredient is a pointwise estimate for primitives of
	the potential difference: their values are identified by asymptotic
	coefficients at infinity, while their size is controlled by
	differences of suitable modified solutions on a fixed annulus of large
	complex frequencies. This yields
	a differential inequality to which a Carleman type unique continuation argument
	applies.
	
	Let us first describe the solutions in question. After extending
	$q$ by zero outside $\Omega$, we consider solutions of
	$(-\Delta+q)\psi_q=0$ on $\C$ of the form
	$$
	\psi_q(z,k)=e^{\i kz}\mu_q(z,k),
	\qquad
	\|\mu_q(\cdot,k)-1\|_{L^\infty(\C)}
	\longrightarrow0
	\quad\text{as $|k|\longrightarrow\infty$}.
	$$
	Here $z=x^1+\i x^2$ is the spatial variable and $k\in\C$ is the
	complex frequency. Nachman's estimates for the Faddeev Green
	function \cite{Nachman1996} give existence and uniqueness of these
	CGO solutions for sufficiently large $|k|$
	by contraction, for arbitrary complex-valued $L^p$ potentials.
	
	The first asymptotic correction to $\mu_q$ identifies a primitive
	of $q$. More precisely, let $\omega_q$ be the convolution of $q$
	with $(2\pi)^{-1}\log|\cdot|$, normalized to vanish at a fixed point
	outside a disk $B$ containing $\overline\Omega$. Then
	$\Delta\omega_q=q$. Reducing to $1<p<2$ and setting
	$\pt=2p/(2-p)$, we prove in Lemma~\ref{lem_firstcoef} that
	$$
	k\bigl(\mu_q(\cdot,k)-1\bigr)
	\longrightarrow-\i\,\partial\omega_q
	\quad\text{in }L^{\pt}(\C),
	\qquad |k|\longrightarrow\infty,
	$$
	where $\partial=(\partial_{x^1}-\i\partial_{x^2})/2$.
	Together with the corresponding expansion for $\overline q$,
	this determines the full spatial gradient of $\omega_q$.
	For two potentials with the same boundary map, exterior agreement
	of the solutions consequently gives agreement of their normalized
	logarithmic potentials near $\partial B$.
	
	The idea behind the annular estimate comes from an elementary
	observation in complex analysis. A bounded holomorphic function
	outside a disk has a limit at infinity, and this limit is controlled
	by its $L^2$ norm on a surrounding annulus. Indeed, inversion
	brings infinity to a removable singularity, after which the
	mean-value identity gives the estimate. We establish a corresponding
	principle for the d-bar equations satisfied by the CGO amplitudes.
	For complex potentials, these equations couple the amplitudes for
	$q$ and $\overline q$; equality of the DN maps gives a common
	homogeneous system for their paired differences. Fixing a
	sufficiently large $R_0$, we work with
	$$
	\mathcal K_{R_0}=\{k\in\C:R_0<|k|<2R_0\}.
	$$
	Under $\kappa=R_0/k$, the relevant equations have coefficients
	small in $L^\tau$ on the unit disk for some $\tau>2$.
	A similarity argument writes each solution as an invertible
	matrix times a holomorphic vector, with uniform bounds for
	the matrix and its inverse. Removability at the origin then
	yields annular control of the asymptotic coefficients, with
	constants independent of the spatial point.
	
	To connect this estimate with the spatial equation, there is
	a further difficulty: the controlled quantity is a primitive
	of the potential difference. Indeed, if
	$v=\psi_{q_1}-\psi_{q_2}$, then
	$$
	(-\Delta+q_1)v=-(q_1-q_2)\psi_{q_2},
	$$
	whereas the asymptotics above identify the gradient of
	$\omega=\omega_{q_1}-\omega_{q_2}$, with $\omega$ satisfying
	$\Delta\omega=q_1-q_2$.
	This motivates the comparison family
	$$
	U(z,k)=
	\begin{pmatrix}
		\psi_{q_1}(z,k)-e^{\omega(z)}\psi_{q_2}(z,k)\\
		\psi_{\overline{q_1}}(z,k)
		-e^{\overline{\omega(z)}}\psi_{\overline{q_2}}(z,k)
	\end{pmatrix}.
	$$
	In its spatial equation, the potential difference cancels
	against $\Delta\omega$, leaving terms linear and quadratic
	in $\nabla\omega$. At the same time, this modification must
	preserve the complex-frequency equation. For complex-valued
	$\omega$, this compatibility is not automatic: multiplication
	by a complex function need not preserve an antilinear d-bar
	equation. The pairing above resolves it: the conjugate
	exponential factors and the off-diagonal structure of the
	coupled system ensure that $e^{-\i kz}U$ satisfies the same
	homogeneous system as the original paired difference.
	
	Applying the annular estimates to the limit of $e^{-\i kz}U$
	at infinity and to the first asymptotic coefficient of the
	original paired amplitude difference gives
	$$
	|\nabla\omega(z)|
	\leq C\left(
	\int_{\mathcal K_{R_0}}|U(z,k)|^2\dA(k)
	\right)^{1/2}
	\quad\text{for almost every $z\in B$}.
	$$
	Thus the primitive identified at infinity is controlled by the
	annular norm of the comparison family at the same spatial point.
	This estimate is derived under equality of the DN maps,
	which supplies the common homogeneous d-bar system. Its pointwise
	form allows us to insert a spatial Carleman weight without
	changing the constant.
	
	Finally, regard $U$ as a function of $z$ with values in
	$\mathcal H=L^2(\mathcal K_{R_0};\C^2)$.
	The primitive estimate controls the inhomogeneous terms
	in its spatial equation and gives
	$$
	\big|\Delta_z U(z)\big|_{\mathcal H}
	\leq V(z)\big|U(z)\big|_{\mathcal H},
	\qquad V\in L^p(B).
	$$
	Here $U\in W^{1,\pt}(B;\mathcal H)$ and
	$\Delta U\in L^p(B;\mathcal H)$.
	The singular-weight Carleman estimate in
	Section~\ref{sec_carleman} applies at precisely this regularity;
	localization allows the term involving $V$ to be absorbed
	for every $p>1$. Exterior agreement makes $U$ vanish near
	$\partial B$, and unique continuation gives $U=0$ throughout
	$B$. The primitive estimate then implies $\nabla\omega=0$,
	and hence $q_1-q_2=\Delta\omega=0$.

	\subsection{Organization of the paper}
	
	Section~\ref{sec_preliminaries} fixes the notations in the paper and also collects the fundamental preliminary results
	on the CGO solutions and the d-bar equation.
	In Section~\ref{sec_primitives}, we construct the primitives of the potentials,
	establish their exterior agreement, and derive the paired
	complex-frequency equations. The fundamental primitive estimate
	is proved in Section~\ref{sec_fundamental}.
	Section~\ref{sec_carleman} establishes the Hilbert space valued
	unique continuation result needed for the proof of
	Theorem~\ref{main_thm}, which is completed in
	Section~\ref{sec_uniqueness}.

	\section{Preliminaries}\label{sec_preliminaries}
	This section collects the preliminary results used throughout
	the paper. We begin by fixing some notation and then introduce CGO solutions and establish their dependence on the complex frequency $k$, the associated
	d-bar equation, and the asymptotic behaviour needed below.
	Our notation and treatment follow Nachman \cite{Nachman1996}
	closely. We work at sufficiently large complex frequencies only,
	where the estimates for the Faddeev Green function yield
	solvability by contraction. We also record explicitly the
	coupling between the equations for $q$ and $\overline q$ that
	arises for complex-valued potentials.
	
	Throughout the paper, $C>0$ denotes a constant that may change
	from line to line. Dependence on particular parameters is
	indicated when relevant. Unless otherwise specified, constants
	in estimates involving the complex frequency $k$ are independent
	of $k$ in the stated range.
	
	\subsection{Cauchy transforms and the Faddeev Green function}
	
	We identify $\R^2$ with $\C$ by $z=x^1+\i x^2$ and set
	$$
	\partial=\frac12\left(\frac{\partial}{\partial x^1}-\i \frac{\partial}{\partial x^2}\right),\qquad
	\bpartial=\frac12\left(\frac{\partial}{\partial x^1}+\i \frac{\partial}{\partial x^2}\right),\qquad
	\Delta=4\bpartial\partial.
	$$
	Fix an open disk $B$ centred at the origin such that 
	\begin{equation}\label{def_ball}
		\overline\Omega\subset B,\end{equation} 
	and extend each potential by zero to $\C$. Since $\Omega$ is bounded, it suffices to work with an exponent $1<p<2$. We fix
	\begin{equation}\label{exponents}
		p'=\frac{p}{p-1},\qquad \pt=\frac{2p}{2-p},\qquad
		\frac{2}{\pt}<\sigma<1.
	\end{equation}
	
	We use the solid Cauchy transforms
	\begin{equation}\label{solid_cauchy}
		\bpartial^{-1}f(z)=\frac1\pi\int_\C\frac{f(w)}{z-w}\dA(w),\qquad
		\partial^{-1}f(z)=\frac1\pi\int_\C\frac{f(w)}{\overline z-\overline w}\dA(w),
	\end{equation}
	where $\dA(w)= \frac{\i}{2}\,dw \wedge d\overline{w}$. They are bounded from $L^p(\C)$ to $L^{\pt}(\C)$, and the Beurling transform $\partial\bpartial^{-1}$ is bounded on $L^p(\C)$.
	
	For $k\in\C\setminus\{0\}$, and any $f\in L^p(\C)$ write
	\begin{equation}\label{T_def}
		e_k(z)=e^{\i (kz+\overline k\overline z)},\qquad
		(\partial + \i k)^{-1} f:=e_{-k}\partial^{-1}(e_k f).
	\end{equation}
	Thus $|e_k|=1$, $(\partial+\i k) (\partial+\i k)^{-1}f=f$, and following \cite[Lemma 1.2(i)]{Nachman1996}, there holds
	\begin{equation}\label{Tkbound}
		\|(\partial+\i k)^{-1} f\|_{L^{\pt}(\C)}\leq C_p\|f\|_{L^p(\C)}.
	\end{equation}
	Let us consider the conjugated Laplacian
	\begin{equation}\label{conj_Laplace}
		P_k=-e^{-\i k z}\Delta (e^{\i kz} \cdot)=-\Delta-4\i k\bpartial.
	\end{equation}
	We recall from \cite[Lemma 1.3]{Nachman1996} that given any $f\in L^{p}(\C)$, the equation 
	\begin{equation}\label{eq_conj_Laplace}
		P_k u = f \quad \text{in $\C$}
	\end{equation}
	admits a unique solution 
	$$u:=g_k*f \in W^{1,\pt}(\C)\cap W^{2,p}_{\mathrm{loc}}(\C)$$ with $g_k$ the so-called Faddeev fundamental solution with Fourier multiplier $(|\xi|^2+2k(\xi_1+\i \xi_2))^{-1}$. We record from \cite[eq. (1.6)]{Nachman1996} that
	\begin{equation}\label{gformula}
		g_k*f=-\frac1{4\i k}\big(\bpartial^{-1}f-(\partial+\i k)^{-1}\partial\bpartial^{-1}f\big).
	\end{equation}
	The Faddeev estimates that we need are those of Nachman \cite[Lemmas~1.2--1.3]{Nachman1996}:
	\begin{equation}
		\|g_k*f\|_{W^{s,\pt}(\C)}\leq C_{p,s}|k|^{s-1}\|f\|_{L^p(\C)}.
		\qquad 0\leq s\leq1,\quad |k|\geq1.\label{gestimate}
	\end{equation}
	We will also use
	\begin{equation}\label{gspatial}
		\bpartial(g_k*f)=-\frac14(\partial+\i k)^{-1} f,\qquad
		\partial(g_k*f)=-\i k(g_k*f)-\frac14\bpartial^{-1}f.
	\end{equation}
	These follow from \eqref{gformula} and are recorded at the end of the proof of \cite[Lemma~2.5, p.~84]{Nachman1996}. Finally, we need \cite[Lemma 2.5]{Nachman1996} but first we introduce some notation for weighted Sobolev spaces. For $\eta\in\R$, we write
	$$
	L^r_\eta(\C)
	=\{f:\langle z\rangle^\eta f\in L^r(\C)\},
	\qquad
	W^{1,r}_\eta(\C)
	=\{f:\langle z\rangle^\eta f\in W^{1,r}(\C)\},
	$$
	with the corresponding weighted norms, where
	$\langle z\rangle=(1+|z|^2)^{1/2}$.
	\begin{lemma}[{\cite[Lemma 2.5]{Nachman1996}}]\label{lem_Nachman_g_k}
		Let $\frac{2}{p'}<\alpha<1$ and $\beta>\frac{2}{\pt}$. The map $k \mapsto g_k*\cdot$ is
		differentiable on $\C \setminus \{0\}$ in the strong operator topology $L^p_\alpha(\C)$ to  $W^{1,\pt}_{-\beta}(\C)$ and
		\begin{equation}\label{gkderivative}
			\frac{\partial}{\partial\overline k}(g_k*f)(z)
			=-\frac{e_{-k}(z)}{4\pi\overline k}\int_\C e_k(w)f(w)\dA(w).
		\end{equation}
	\end{lemma}
	\subsection{Large-frequency complex geometric optics}
	We assume throughout that $q\in L^p(\C)$, for some $p\in (1,2)$ and it is supported in $\overline\Omega$.
	\begin{lemma}\label{lem_cgo}
		There is $k_0\geq1$ such that, for any $k\in \C$ with $|k| >k_0$, there exists a unique solution to the equation 
		$$
		(-\Delta +q) \psi_q(\cdot,k) =0 \quad \text{in $\mathcal D'(\C)$}
		$$
		with $e^{-\i kz}\psi_q(\cdot,k)-1 \in L^{\pt}\cap L^\infty$. For each fixed $k$ with $|k|>k_0$, the function 
		$$
		\mu_q(z,k):= e^{-\i kz}\psi_q(z,k) \quad z\in \C
		$$
		satisfies the bounds
		\begin{align}
			\|\mu_q(\cdot,k)-1\|_{L^\infty(\C)}&\leq C|k|^{\sigma-1},\label{muinfty}\\
			\|\mu_q(\cdot,k)-1\|_{W^{1,\pt}(\C)}
			+|k|\|\mu_q(\cdot,k)-1\|_{L^{\pt}(\C)}&\leq C,\label{muSobolev}
		\end{align}
		for some $C>1$ that is independent of $k$.
	\end{lemma}

	\begin{proof}
		For $|k|\geq1$, define $\mathcal T_kv=g_k*(qv)$. Recalling \eqref{exponents}, it follows from
		\eqref{gestimate} and the embedding
		$W^{\sigma,\pt}(\C)\hookrightarrow L^\infty(\C)$, that 
		$$
		\|\mathcal T_k\|_{L^\infty(\C)\to L^\infty(\C)}
		\leq C_0|k|^{\sigma-1}\|q\|_{L^p(\C)},
		$$
		where $C_0$ depends only on $p$ and $\sigma$.
		Choose $k_0\geq1$ sufficiently large that
		\begin{equation}\label{def_k_0}
			C_0k_0^{\sigma-1}\|q\|_{L^p(\C)}\leq\frac12.
		\end{equation}
		For $|k|>k_0$, the operator $I+\mathcal T_k$ is therefore
		invertible on $L^\infty(\C)$, with inverse norm at most $2$.
		Consequently, the equation
		\begin{equation}\label{LS}
			\mu_q(\cdot,k)
			=1-g_k*\bigl(q\mu_q(\cdot,k)\bigr)
		\end{equation}
		has a unique bounded solution, satisfying
		$\|\mu_q(\cdot,k)\|_{L^\infty(\C)}\leq2$.
		
		Applying \eqref{gestimate} to \eqref{LS} gives
		$$
		\|\mu_q(\cdot,k)-1\|_{W^{s,\pt}(\C)}
		\leq 2C_{p,s}|k|^{s-1}\|q\|_{L^p(\C)},
		\qquad 0\leq s\leq1.
		$$
		In particular, $\mu_q(\cdot,k)-1\in W^{1,\pt}(\C)$.
		Moreover, \eqref{LS} implies
		$P_k(\mu_q-1)=-q\mu_q$, so \eqref{conj_Laplace} shows that
		$\psi_q(z,k)=e^{\i kz}\mu_q(z,k)$ solves
		$(-\Delta+q)\psi_q=0$ distributionally.
		The choices $s=\sigma,0,1$, together with the Sobolev embedding,
		yield \eqref{muinfty} and \eqref{muSobolev}.
		
		Finally, let $\psi$ be any solution in the stated class and put
		$\mu=e^{-\i kz}\psi$. Then
		$$
		P_k(\mu-1)=-q\mu,
		\qquad
		\mu-1\in L^{\pt}(\C),
		\qquad
		q\mu\in L^p(\C).
		$$
		Uniqueness for the free equation in $L^{\pt}(\C)$
		\cite[Lemma~1.3]{Nachman1996} gives
		$\mu-1=-g_k*(q\mu)$. Thus
		$(I+\mathcal T_k)\mu=1$, and invertibility on $L^\infty(\C)$
		implies $\mu=\mu_q$. This proves uniqueness.
	\end{proof}

	\begin{lemma}\label{lem_kregularity}
		Let $B$ be as in \eqref{def_ball}. The map
		$$
		k\longmapsto\mu_q(\cdot,k)|_{B}
		$$
		is continuously differentiable from $\{|k|>k_0\}\subset\C$ to $W^{1,\pt}(B)$.
	\end{lemma}
	
	\begin{proof}
		Set $X=W^{1,\pt}(B)$ and define an operator on $X$ as follows,
		$$
		K_q(k):=G_B(k)M_q,
		$$
		where
		$$
		M_qv=q(v|_\Omega),
		\qquad
		G_B(k)f=(g_k*\widetilde f)|_B,
		$$
		and $\widetilde f$ denotes extension by zero to $\C$. Nachman's Lemma~2.5 and its proof
		\cite{Nachman1996} show that
		$$
		k\longmapsto G_B(k)
		$$
		depends strongly $C^1$ on $k\ne0$ as an operator from
		$L^p(\Omega)$ to $X$. Since $\pt>2$, the embedding
		$X\hookrightarrow C(\overline B)$ is compact, and hence
		$M_q:X\to L^p(\Omega)$ is compact. Composition on the right with
		a compact operator upgrades strong $C^1$ dependence to $C^1$
		dependence in operator norm; this is also the argument used in
		\cite[proof of Theorem~2.1]{Nachman1996}. Consequently,
		$K_q(k)$ is compact on $X$ and depends $C^1$ on $k$ in operator
		norm.
		
		For $|k|>k_0$, the contraction estimate in the proof of
		Lemma~\ref{lem_cgo} shows that $I+K_q(k)$ is injective on $X$.
		It is therefore invertible by the Fredholm alternative. Since
		$$
		\mu_q(\cdot,k)|_B=(I+K_q(k))^{-1}1,
		$$
		the conclusion follows from the smoothness of inversion on the
		bounded invertible operators on $X$.
	\end{proof}
	
	\begin{remark}\label{remark_mu_point}
		In particular, evaluation at any $z\in\overline{B}$ is a continuous linear functional on the space in Lemma~\ref{lem_kregularity}. Thus $\mu_q(z,k)$ is continuously differentiable in $k$ for each such $z$.\end{remark}
	
	Next, we record the well-known large-frequency d-bar equation satisfied by the CGO solutions above. Before stating the lemma, we recall the notion of the {\em scattering} transform of $q$ as follows,
	\begin{equation}\label{tdef}
		t_q(k)=\int_\C e_k(z)q(z)\mu_q(z,k)\dA(z) \qquad |k|>k_0.
	\end{equation}
	
	\begin{lemma}\label{lem_complex_Dbar}
		For $|k|>k_0$,
		\begin{equation}\label{complex_Dbar}
			\frac{\partial\mu_q}{\partial\overline k}(z,k)
			=\frac{t_q(k)}{4\pi\overline k}\,e_{-k}(z)\,
			\overline{\mu_{\overline q}(z,k)},
		\end{equation}
		The identity holds in $W^{1,\pt}(B)$ and hence also pointwise for every $z\in \overline{B}$.
	\end{lemma}
	
	\begin{proof}
		The choice \eqref{def_k_0} applies equally to $\overline q$.
		With $K_q(k)$ as in the preceding proof and $t_q$ as in
		\eqref{tdef}, differentiation of \eqref{LS}, justified by
		Lemma~\ref{lem_kregularity} and \eqref{gkderivative}, yields
		\begin{equation}\label{differentiatedLS}
			(I+K_q(k))\frac{\partial\mu_q}{\partial\overline k}
			=\frac{t_q(k)}{4\pi\overline k}e_{-k}
			\quad\text{in }W^{1,\pt}(B).
		\end{equation}
		For $f\in L^p(\C)$, the free Green operator satisfies
		\begin{equation}\label{kernel_symmetry}
			g_k*(e_{-k}\overline f)=e_{-k}\overline{g_k*f}.
		\end{equation}
		Indeed, $P_k(e_{-k}\overline v)=e_{-k}\overline{P_kv}$
		in distributions; applying this to $v=g_k*f$ and using
		uniqueness in \cite[Lemma 1.3]{Nachman1996} proves \eqref{kernel_symmetry}.
		Conjugating \eqref{LS} for $\overline q$ and using
		\eqref{kernel_symmetry} with $f=\overline q\mu_{\overline q}$
		therefore gives
		$$
		(I+K_q(k))\bigl(e_{-k}\overline{\mu_{\overline q}}\bigr)
		=e_{-k}
		\quad\text{in $W^{1,\pt}(B)$}.
		$$
		Invertibility of $I+K_q(k)$ in \eqref{differentiatedLS}
		proves \eqref{complex_Dbar}. 
	\end{proof}

	We will need the following large frequency asymptotic lemma. 
	\begin{lemma}\label{lem_firstcoef}
		Let us define the primitive $A_q=\frac14\bpartial^{-1}q$. Then
		\begin{equation}\label{firstcoef}
			k\big(\mu_q(\cdot,k)-1\big)\longrightarrow-\i A_q
			\quad\text{in }L^{\pt}(\C),\qquad \text{as $|k|\longrightarrow\infty$}.
		\end{equation}
	\end{lemma}
	
	\begin{proof}
		We claim that
		\begin{equation}\label{Tkvanishes}
			(\partial+\i k)^{-1} f\longrightarrow0\quad\text{in }L^{\pt}(\C)
			\quad\text{for every }f\in L^p(\C).
		\end{equation}
		Indeed, for $f\in C^\infty_c(\C)$, the inverse identity for $\partial+\i k$ gives
		$$
		(\partial+\i k)^{-1} f=\frac1{\i k}\big(f-(\partial+\i k)^{-1}\partial f\big),
		$$
		and thus \eqref{Tkbound} proves \eqref{Tkvanishes} for smooth compactly supported functions. Density and the uniform bound \eqref{Tkbound} extend the conclusion to $f\in L^p(\C)$ as claimed. Therefore \eqref{gformula} implies
		$$
		k(g_k*f)\longrightarrow\frac{\i }{4}\bpartial^{-1}f
		\quad\text{in }L^{\pt}(\C) \quad \text{as $|k|\to \infty$}
		$$
		for each fixed $f\in L^p(\C)$. On the other hand, \eqref{muinfty} gives
		$$
		\|q(\mu_q(\cdot,k)-1)\|_{L^p}\longrightarrow0  \quad \text{as $|k|\to \infty$},
		$$
		and note also that $k\,g_k*$ is uniformly bounded from $L^p$ to $L^{\pt}$ by \eqref{gestimate}. Substitution in \eqref{LS} proves \eqref{firstcoef}.
	\end{proof}

	\subsection{Boundary determination at large complex frequencies}
	
	We continue to assume that $q\in L^p(\C)$ is supported in $\overline\Omega$. Following the notation in \cite{Nachman1996}, let $S_k$ denote the operator
	$$
	S_kf(x)= \int_{\partial\Omega} G_k(x-y)f(y)\,d\sigma(y),
	$$
	the single-layer operator corresponding to the Green function $G_k=e^{\i kz}g_k$ that can also be defined distributionally as
	$$
	G_k(x)= \frac{e^{\i k(x^1+\i x^2)}}{(2\pi)^2} \int_{\R^2} \frac{e^{\i \xi\cdot x}}{|\xi|^2+2k(\xi_1+\i\xi_2)}\,d\xi.
	$$
	It is bounded from $H^{-1/2}(\partial\Omega)$ to $H^{1/2}(\partial\Omega)$; see \cite[Lemma~7.1]{Nachman1996}.
	
	\begin{lemma}\label{lem_boundary}
		Suppose that zero is not a Dirichlet eigenvalue for $-\Delta+q$ on $\Omega$. For $|k|>k_0$, the trace $h_q=\psi_q|_{\partial\Omega}$ is the unique solution in $H^{1/2}(\partial\Omega)$ of
		\begin{equation}\label{boundaryintegral}
			\big[I+S_k(\Lambda_q-\Lambda_0)\big]h_q=e^{\i kz}|_{\partial\Omega}.
		\end{equation}
		Furthermore,
		\begin{align}
			\psi_q(z,k)&=e^{\i kz}-S_k\big((\Lambda_q-\Lambda_0)h_q\big)(z),
			\quad z\in\C\setminus\overline\Omega,\label{exteriorpsi}\\
			t_q(k)&=\big\langle(\Lambda_q-\Lambda_0)h_q,e^{\i \overline k\overline z}|_{\partial\Omega}\big\rangle.\label{boundaryt}
		\end{align}
		Consequently, the knowledge of the DN map associated to a potential $q\in L^p(\Omega)$, $p>1$, uniquely determines $\psi_q$ and $\psi_{\overline{q}}$ outside $\overline\Omega$ and also uniquely determines $t_q$ and $t_{\overline{q}}$ for all $|k|>k_0$. 
	\end{lemma}
	
	\begin{proof}
		These are \cite[Theorem~5, equations~(0.17), (0.18), and~(7.3)]{Nachman1996}. That theorem applies to $L^p$ potentials, $p>1$, at every nonexceptional complex frequency and does not impose that $q$ be real-valued. Lemma~\ref{lem_cgo} supplies nonexceptional frequencies for $|k|>k_0$. Equations~\eqref{boundaryintegral}--\eqref{boundaryt} therefore give the assertions for $q$. The same conclusions hold for $\overline q$, since zero is not a
		Dirichlet eigenvalue for $-\Delta+\overline q$ and
		$\Lambda_{\overline q}$ is determined by $\Lambda_q$ through
		\begin{equation}\label{conjdn}
			\Lambda_{\overline q}f=\overline{\Lambda_q\overline f},
		\end{equation}
	\end{proof}
	
	\section{Primitives and the complex-frequency equations}\label{sec_primitives}
	
	From now on, $q_1,q_2$ satisfy the hypotheses of Theorem~\ref{main_thm}, and
	\begin{equation}\label{equalDN}
		\Lambda_{q_1}=\Lambda_{q_2}.
	\end{equation}
	We identify $q_1,q_2$ as functions in $L^p(\C)$ by setting them to be zero outside $\Omega$. We also recall that by reducing $p$ if necessary, we may assume without loss of generality that $1<p<2$. We fix $k_0\geq 1$ so that \eqref{def_k_0} is satisfied for both $q_1$, $q_2$, and introduce the notations
	\begin{equation}\label{mu_j}
		\mu_j=\mu_{q_j},\qquad \mu_j^c=\mu_{\overline{q_j}},\qquad
		\psi_j=e^{\i kz}\mu_j,\qquad \psi_j^c=e^{\i kz}\mu_j^c.
	\end{equation}
	The superscript $c$ refers to the conjugate potential; it does not denote conjugation of the solution.
	
	Fix $z_*\in\C\setminus\overline{B}$ and define the following primitives on $\C$,
	\begin{equation}\label{primitive_def}
		\begin{aligned}
			A_j&=\frac14\bpartial^{-1}q_j,\qquad
			A_j^c=\frac14\bpartial^{-1}\overline{q_j},\\
			\omega_j(z)&=\frac1{2\pi}\int_\Omega
			\log\frac{|z-w|}{|z_*-w|}\,q_j(w)\dA(w).
		\end{aligned}
	\end{equation}
	The natural logarithm is locally in every finite $L^a$ space, so the last integral is well defined and $\omega_j\in L^\infty_\loc(\C)$. Since $\Delta\omega_j=q_j$ distributionally, elliptic regularity gives $\omega_j\in W^{2,p}_\loc(\C)$. The Sobolev embedding makes $\omega_j$ continuous. Differentiation of the logarithmic kernel and \eqref{solid_cauchy} give
	\begin{equation}\label{primitive_iden}
		\partial\omega_j=A_j,\qquad
		\bpartial\omega_j=\overline{A_j^c},\qquad
		\Delta\omega_j=q_j,\qquad \omega_j(z_*)=0.
	\end{equation}
	In particular, $\nabla\omega_j\in L^{\pt}(\C)$. Set
	\begin{equation}\label{primitive_diff}
		\omega=\omega_1-\omega_2,\qquad A=A_1-A_2,\qquad A^c=A_1^c-A_2^c.
	\end{equation}
	The Euclidean norm of the gradient satisfies
	\begin{equation}\label{gradientNorm}
		|\nabla\omega|^2=2\big(|A|^2+|A^c|^2\big).
	\end{equation}
	
	\begin{proposition}\label{prop_exterior_primitives}
		The functions $A_1,A_2$ agree on $\C\setminus\overline\Omega$, as do $A_1^c,A_2^c$. Moreover, $\omega_1=\omega_2$ on the unbounded component of $\C\setminus\overline\Omega$. In particular,
		\begin{equation}\label{compactPrimitive}
			\supp\omega\subset B.
		\end{equation}
	\end{proposition}
	
	\begin{proof}
		Lemmas~\ref{lem_boundary} and~\ref{lem_firstcoef} imply
		$$
		A=A^c=0\quad\text{almost everywhere on }\C\setminus\overline\Omega.
		$$
		These functions are smooth there, so the identities hold pointwise. Equation~\eqref{primitive_iden} gives $\partial\omega=\bpartial\omega=0$ on this open set. Hence $\omega$ is constant on each of its connected components. Its value on the unbounded component is zero because that component contains $z_*$ and $\omega(z_*)=0$. This proves \eqref{compactPrimitive}.
	\end{proof}

	By Lemma~\ref{lem_boundary}, the functions
	\begin{equation}\label{t_equal}
		t=t_{q_1}=t_{q_2},\qquad t^c=t_{\overline{q_1}}=t_{\overline{q_2}}
	\end{equation}
	are defined for $|k|>k_0$ where we recall that $k_0$ is fixed sufficiently large so that \eqref{def_k_0} is satisfied for both potentials. For $|k|>k_0$ and $z\in \C$ we introduce 
	\begin{equation}\label{pairedDefinitions}
		\bmu_j=\begin{pmatrix}\mu_j\\\mu_j^c\end{pmatrix},\qquad
		\bpsi_j=e^{\i kz}\bmu_j,\qquad
		\mathcal B(z,k)=\frac{e_{-k}(z)}{4\pi\overline k}
		\begin{pmatrix}0&t(k)\\t^c(k)&0\end{pmatrix}.
	\end{equation}
	Lemma~\ref{lem_complex_Dbar} is equivalent to
	\begin{equation}\label{pairedDbar}
		\frac{\partial}{\partial\overline k}\bmu_j=\mathcal B\overline{\bmu_j},
	\end{equation}
	for all $|k|>k_0$ and $z\in \overline{B}$. For $j=1,2$, define
	\begin{equation}\label{Epsi}
		\mathcal W_j(z)=\diag\big(e^{-\omega_j(z)},e^{-\overline{\omega_j(z)}}\big),\qquad
		\widetilde{\bmu}_j=\mathcal W_j\bmu_j,\qquad
		\widetilde{\bpsi}_j=\mathcal W_j\bpsi_j.
	\end{equation}
	Clearly,
	\begin{equation}\label{intertwining}
		\mathcal W_j\mathcal B=\mathcal B\overline{\mathcal W_j}.
	\end{equation}
	For $|k|>k_0$ and $z\in \C$, define
	\begin{equation}\label{difference_def}
		\bmu=\bmu_1-\bmu_2,\qquad
		\widetilde{\bmu}=\widetilde{\bmu}_1-\widetilde{\bmu}_2,\qquad
		\widetilde{\bpsi}=\widetilde{\bpsi}_1-\widetilde{\bpsi}_2
		=e^{\i kz}\widetilde{\bmu},\qquad
		c=e^{-\omega_1}-e^{-\omega_2}.
	\end{equation}
	
	\begin{proposition}\label{prop_paired_diff}
		For $|k|>k_0$ and $z\in \overline{B}$, there holds
		\begin{equation}\label{diff_Dbar}
			\frac{\partial}{\partial\overline k}\bmu=\mathcal B\overline{\bmu},\qquad
			\frac{\partial}{\partial\overline k}\widetilde{\bmu}
			=\mathcal B\overline{\widetilde{\bmu}}.
		\end{equation}
		There is a compact subset of $B$, independent of $k$, containing the supports of $c$, $\widetilde{\bmu}(\cdot,k)$, and $\widetilde{\bpsi}(\cdot,k)$. As $|k|\to\infty$,
		\begin{align}
			\widetilde{\bmu}(\cdot,k)&\longrightarrow
			\begin{pmatrix}c\\\overline c\end{pmatrix}
			\quad\text{uniformly on }\overline{B},\label{tildeLimit}\\
			k\bmu(\cdot,k)&\longrightarrow-\i \begin{pmatrix}A\\A^c\end{pmatrix}
			\quad\text{in }L^{\pt}(\C;\C^2).\label{delta_lim}
		\end{align}
		The error in \eqref{tildeLimit} is $O(|k|^{\sigma-1})$ with $\sigma$ as in \eqref{exponents}, and
		\begin{equation}\label{delta_bound}
			\sup_{|k|>k_0}\|k\bmu(\cdot,k)\|_{L^{\pt}(\C)}<\infty.
		\end{equation}
	\end{proposition}
	
	\begin{proof}
		Subtracting \eqref{pairedDbar} proves the equation for $\bmu$. Since $\mathcal W_j$ is independent of $k$, \eqref{intertwining} shows that $\widetilde{\bmu}_j$ satisfies the same equation as $\bmu_j$. Subtraction gives the second identity in \eqref{diff_Dbar}. On the unbounded component of $\C\setminus\overline\Omega$, Lemma~\ref{lem_boundary} gives $\bmu_1=\bmu_2$, and Proposition~\ref{prop_exterior_primitives} gives $\mathcal W_1=\mathcal W_2$. The same neighborhood of $\partial B$ therefore works for all the asserted supports. Finally, \eqref{tildeLimit} follows from \eqref{muinfty} and boundedness of $\mathcal W_j$ on $\overline{B}$, while \eqref{delta_lim} and \eqref{delta_bound} follow from Lemma~\ref{lem_firstcoef} and \eqref{muSobolev}.
	\end{proof}

	\section{Fundamental primitive estimate}\label{sec_fundamental}
	
	The main aim of this section is to prove the following primitive estimates.
	
	\begin{proposition}[The primitive estimate]\label{prop_fundamental}
		Let $q_1,q_2\in L^p(\Omega;\C)$, $p \in (1,2)$, be as in the hypothesis of Theorem~\ref{main_thm}, and assume that
		$$
		\Lambda_{q_1}=\Lambda_{q_2}.
		$$
		Let $B$ be as in \eqref{def_ball} and fix $z_*\in \C \setminus \overline{B}$. Let the functions $\omega$, $c$ and $\widetilde{\bpsi}$ be as in \eqref{primitive_diff} and \eqref{difference_def}. Let $R_0>k_0$ be fixed as in \eqref{chooseR}. Then there exists $C\geq 1$ such that,
		\begin{equation}\label{fundamental_psi}
			|\nabla\omega(z)|+|c(z)|\leq C \left(\int_{\mathcal K_{R_0}}|\widetilde{\bpsi}(z,k)|^2\dA(k)\right)^{1/2}
			\quad\text{for a.e. $z\in\C$}.
		\end{equation}
	\end{proposition}
	
	Let us emphasize that the key content of \eqref{fundamental_psi} lies in the region $z\in B$ as all of the functions involved are compactly supported in $B$. Before proving the proposition, we need to fix some notation and prove some auxiliary lemmas. We continue to identify functions in $L^p(\Omega)$ with their extensions by zero to $\C$. Let
	\begin{equation}\label{disk_annulus}
		\D=\{\kappa\in\C:|\kappa|<1\},\qquad
		\mathcal E=\{\kappa\in\C:1/2<|\kappa|<1\}.
	\end{equation}
	All vector norms below are Euclidean, and matrix norms are the corresponding operator norms. We fix an exponent
	\begin{equation}\label{tauchoice}
		2<\tau<\min\{p',\pt,2/\sigma\}.
	\end{equation}
	Such an exponent exists by \eqref{exponents}.
	
	\begin{lemma}\label{lem_small_coeff}
		For $R>k_0$, $z\in\C$, and $0<|\kappa|<1$, define
		\begin{align}
			Q_{0,z,R}(\kappa)
			&=-\frac{e_{-R/\kappa}(z)}{4\pi\overline\kappa}
			\begin{pmatrix}0&t(R/\kappa)\\t^c(R/\kappa)&0\end{pmatrix},\label{Qzero}\\
			Q_{1,z,R}(\kappa)
			&=-\frac{e_{-R/\kappa}(z)}{4\pi\kappa}
			\begin{pmatrix}0&t(R/\kappa)\\t^c(R/\kappa)&0\end{pmatrix},\label{Qone}
		\end{align}
		where $t$ and $t^c$ are as in \eqref{t_equal}. Then, for $\ell=0,1$,
		\begin{equation}\label{Qsmall}
			\sup_{z\in\C}\|Q_{\ell,z,R}\|_{L^\tau(\D)}
			\leq C\big(R^{-\frac{2}{p'}}+R^{\sigma-1}\big)
			\longrightarrow0\qquad\text{as }R\longrightarrow\infty.
		\end{equation}
	\end{lemma}
	
	\begin{proof}
		We prove the claim only for $Q_{0,z,R}$ as the claim for $Q_{1,z,R}$ follows analogously. Recalling \eqref{t_equal}, we write for each $|k|>k_0$,
		\begin{equation}\label{tdecomposition}
			t(k)=T(k)+E(k),\qquad
			T(k)=\int_\C e_k(z)q_1(z)\dA(z).
		\end{equation}
		The function $T(k)$ is the Fourier transform of $q_1$ evaluated at $(-2\mathrm{Re}\, k, 2\mathrm{Im}\,k)$. The Hausdorff--Young inequality therefore gives $T\in L^{p'}(\C)$, while \eqref{muinfty} gives
		\begin{equation}\label{tError}
			|E(k)|\leq\|q_1\|_{L^1}\|\mu_{q_1}(\cdot,k)-1\|_{L^\infty}
			\leq C|k|^{\sigma-1}.
		\end{equation}
		The same decomposition applies to $t^c$, with $q_1$ replaced by $\overline{q_1}$, and all of the estimates below hold for $t^c$ as well. Under $k=R/\kappa$, the area element transforms as $\dA(\kappa)=R^2|k|^{-4}\dA(k)$. Since $|e_{-k}(z)|=1$, we obtain
		\begin{equation}\label{Qintegral}
			\|Q_{0,z,R}\|_{L^\tau(\D)}^\tau
			\leq\frac{R^{2-\tau}}{(4\pi)^\tau}
			\int_{|k|>R}\big(|t(k)|^\tau+|t^c(k)|^\tau\big)|k|^{\tau-4}\dA(k).
		\end{equation}
		We proceed to estimate the right hand side by using the decomposition \eqref{tdecomposition}. For the Fourier term $T(k)$ appearing in the right hand side, H\"older's inequality yields
		\begin{align*}
			\int_{|k|>R}|T(k)|^\tau|k|^{\tau-4}\dA(k)
			&\leq\|T\|_{L^{p'}}^\tau
			\left(\int_{|k|>R}|k|^{(\tau-4)p'/(p'-\tau)}\dA(k)\right)^{1-\tau/p'}\\
			&\leq C\|T\|_{L^{p'}}^\tau R^{\tau-2-2\tau/p'} \leq C\,R^{\tau-2-2\tau/p'}
		\end{align*}
		The radial integral converges because $\tau<\pt$, which is equivalent to $$(\tau-4)p'/(p'-\tau)<-2.$$ Next, \eqref{tError} gives
		$$
		\int_{|k|>R}|E(k)|^\tau|k|^{\tau-4}\dA(k)
		\leq C\int_{|k|>R}|k|^{\tau\sigma-4}\dA(k)
		\leq C R^{\tau\sigma-2},
		$$
		where we used $\tau\sigma<2$. Substitution of the previous two bounds into \eqref{Qintegral} proves \eqref{Qsmall}.
	\end{proof}
	
	\begin{lemma}\label{lem_matrix_sim}
		Let $\tau$ be fixed as in \eqref{tauchoice}. There are constants $\varepsilon\in (0,1)$ and $C\geq 1$ with the following property: suppose that $Q\in L^\tau(\D;\C^{2\times2})$, $\|Q\|_{L^\tau}\leq\varepsilon$, and
		$$
		u\in W^{1,\tau}_\loc(\D\setminus\{0\};\C^2)\cap L^2(\D;\C^2),\qquad
		\bpartial u=Q\overline u\quad\text{in }\D\setminus\{0\}.
		$$
		Then $u$ has a continuous extension to $\D$, and
		\begin{equation}\label{ann_tr}
			|u(0)|\leq C\|u\|_{L^2(\mathcal E)},
		\end{equation}
		where we recall that $\mathcal E$ is given by \eqref{disk_annulus}.
	\end{lemma}
	
	\begin{proof}
		By the Sobolev embedding theorem, $u$ is continuous on $\D \setminus \{0\}$. Let us define a matrix-valued function $\widetilde Q$ on $\C$ as follows. For all $\kappa\in \mathbb D \setminus \{0\}$ with $u(\kappa)\ne0$, set
		$$
		\widetilde Q(\kappa)=\frac{(Q(\kappa)\overline{u(\kappa)})u(\kappa)^*}{|u(\kappa)|^2},
		$$
		where ${}^*$ denotes conjugate transpose, and set $\widetilde Q(\kappa)=0$ everywhere else. Then
		\begin{equation}\label{Q_est}
			\widetilde Qu=Q\overline u,\qquad |	\widetilde Q|\leq|Q|,\qquad
			\|\widetilde Q\|_{L^\tau(\C)}\leq\varepsilon.
		\end{equation}
		For functions on $\C$ that are supported in $\overline\D$, the solid Cauchy transform satisfies
		\begin{equation}\label{local_cauchy_bound}
			\|\bpartial^{-1}f\|_{L^\infty(\C)}\leq c_0\|f\|_{L^\tau(\D)},
		\end{equation}
		for some $c_0>0$. Indeed, H\"older's inequality applies to $|\kappa-w|^{-1}$ on $\D$, uniformly in $\kappa$, since $\tau/(\tau-1)<2$.
		
		Choose $\varepsilon \in (0,1)$ sufficiently small so that $c_0\varepsilon\leq1/4$. As $\widetilde Q$ is supported in $\overline{\mathbb D}$, it follows from \eqref{Q_est}-\eqref{local_cauchy_bound} that the equation
		$$
		P=I+\bpartial^{-1}(	\widetilde QP)
		$$
		has a unique bounded matrix solution $P \in L^\infty(\C;\C^{2\times 2})$ by contraction, and
		\begin{equation}\label{Pbounds}
			\|P-I\|_{L^\infty}\leq\frac13,\qquad
			\|P\|_{L^\infty}\leq\frac43,\qquad
			\|P^{-1}\|_{L^\infty}\leq\frac32.
		\end{equation}
		The Cauchy transform and the $L^\tau$ boundedness of its first derivatives imply $P\in W^{1,\tau}_\loc(\C)$. In particular, $P$ is continuous, is invertible at every point, and satisfies 
		$$\bpartial P=	\widetilde QP.$$
		
		It follows that $h=P^{-1}u$ satisfies $\bpartial h=0$ on $\D\setminus\{0\}$. It is therefore holomorphic there. By \eqref{Pbounds}, $h\in L^2(\D)$. Therefore the singularity at zero is removable and $h$ is holomorphic on $\D$, and $u=Ph$ extends continuously across zero. The mean-value identity, integrated over the circles in $\mathcal E$, gives
		$$
		|h(0)|\leq |\mathcal E|^{-1/2}\|h\|_{L^2(\mathcal E)}.
		$$
		Combining this with \eqref{Pbounds} proves \eqref{ann_tr}.
	\end{proof}
	
	In view of Lemma~\ref{lem_small_coeff}, and with $\varepsilon \in (0,1)$ fixed as in Lemma~\ref{lem_matrix_sim}, let us now fix a sufficiently large $R_0>k_0$ so that
	\begin{equation}\label{chooseR}
		\sup_{z\in\C}\|Q_{\ell,z,R_0}\|_{L^\tau(\D)}\leq\varepsilon,
		\qquad \quad \ell=0,1.
	\end{equation}
	Let
	\begin{equation}\label{frequency_ann}
		\mathcal K_{R_0}=\left\{k\in\C\,:\,R_0<|k|<2R_0\right\}.
	\end{equation}
	
	\begin{proposition}\label{prop_twoTraces}
		There exists $C\geq 1$ such that
		\begin{align}
			|c(z)|
			\leq C\|\widetilde{\bmu}(z,\cdot)\|_{L^2(\mathcal K_{R_0})}\label{firstTrace}
		\end{align}
		for all $z\in \overline{B}$ and 
		\begin{align}
			|A(z)|+|A^c(z)|
			\leq C\|\bmu(z,\cdot)\|_{L^2(\mathcal K_{R_0})}\label{ATrace}
		\end{align}
		for almost every $z\in B$.
	\end{proposition}
	
	\begin{proof}
		For $z\in\overline B$ and $0<|\kappa|<1$, set
		$$
		u_0(z,\kappa)
		=\widetilde{\bmu}(z,R_0/\kappa),
		\qquad
		u_1(z,\kappa)
		=\frac{R_0}{\kappa}\bmu(z,R_0/\kappa).
		$$
		By Lemma~\ref{lem_kregularity}, these functions are continuously
		differentiable in $\kappa$ away from zero. The chain rule and
		\eqref{diff_Dbar} give
		$$
		\frac{\partial u_\ell}{\partial\overline\kappa}
		=Q_{\ell,z,R_0}\overline{u_\ell},
		\qquad \ell=0,1,
		\quad\text{in $\D\setminus\{0\}$}.
		$$
		Moreover, \eqref{muinfty} and the boundedness of $\mathcal W_j$
		on $\overline B$ imply
		$$
		|u_0(z,\kappa)|\leq C,
		\qquad
		|u_1(z,\kappa)|
		\leq C R_0^\sigma|\kappa|^{-\sigma},
		$$
		uniformly in $z\in\overline B$. Since $\sigma<1$, both functions
		belong to $L^2(\D;\C^2)$ for every fixed $z$.
		Consequently, \eqref{chooseR} and Lemma~\ref{lem_matrix_sim}
		give continuous extensions across zero and
		$$
		|u_\ell(z,0)|
		\leq C\|u_\ell(z,\cdot)\|_{L^2(\mathcal E)},
		\qquad \ell=0,1,
		$$
		with $C$ independent of $z$.
		
		By \eqref{tildeLimit},
		$$
		u_0(z,0)=
		\left(c(z),\overline{c(z)}\right)^{\mathsf T},
		\qquad z\in\overline B.
		$$
		To identify the second value, use the convergence in
		\eqref{delta_lim} to choose a sequence $|k_m|>R_0$,
		$|k_m|\to\infty$, such that
		$$
		k_m\bmu(z,k_m)
		\longrightarrow
		-\i\begin{pmatrix}A(z)\\ A^c(z)\end{pmatrix}
		\quad\text{for almost every }z\in B.
		$$
		Since $R_0/k_m\to0$, continuity of the extension yields
		$$
		u_1(z,0)
		=\lim_{m\to\infty}u_1(z,R_0/k_m)
		=-\i\begin{pmatrix}A(z)\\ A^c(z)\end{pmatrix}
		\quad\text{for almost every }z\in B.
		$$
		
		Finally, the change of variables $k=R_0/\kappa$ gives
		\begin{align*}
			\|u_0(z,\cdot)\|_{L^2(\mathcal E)}^2&=R_0^2\int_{\mathcal K_{R_0}}
			|\widetilde{\bmu}(z,k)|^2|k|^{-4}\dA(k)\leq R_0^{-2}
			\|\widetilde{\bmu}(z,\cdot)\|_{L^2(\mathcal K_{R_0})}^2,
			\\
			\|u_1(z,\cdot)\|_{L^2(\mathcal E)}^2&=R_0^2\int_{\mathcal K_{R_0}}
			|\bmu(z,k)|^2|k|^{-2}\dA(k)\leq
			\|\bmu(z,\cdot)\|_{L^2(\mathcal K_{R_0})}^2.
		\end{align*}
		Combining these inequalities with the bounds for
		$u_\ell(z,0)$ proves \eqref{firstTrace} and \eqref{ATrace}.
	\end{proof}	
	
	We are now ready to prove the main primitive estimate. 
	
	\begin{proof}[Proof of Proposition~\ref{prop_fundamental}]
		We need to transfer the right hand side of \eqref{ATrace} to $\widetilde{\bmu}$. Recalling \eqref{difference_def}, we use the identity
		\begin{equation}\label{transferIdentity}
			\bmu=\mathcal W_1^{-1}\widetilde{\bmu}
			+(\mathcal W_1^{-1}-\mathcal W_2^{-1})\widetilde{\bmu}_2.
		\end{equation}
		Since
		$$
		e^{\omega_1}-e^{\omega_2}=-e^{\omega_1+\omega_2}c,
		$$
		boundedness of $\omega_j$ on $\overline{B}$ gives
		$$
		|\mathcal W_1^{-1}(z)|\leq C,\qquad |\mathcal W_1^{-1}(z)-\mathcal W_2^{-1}(z)|\leq C|c(z)| \qquad z\in B,
		$$
		for some $C>0$. The functions $\widetilde{\bmu}_2$ are uniformly bounded on $\overline{B}\times\{|k|>k_0\}$ by \eqref{muinfty}. Taking the $L^2(\mathcal K_{R_0})$ norm of \eqref{transferIdentity} we obtain
		\begin{align*}
			\|\bmu(z,\cdot)\|_{L^2(\mathcal K_{R_0})}
			&\leq C\big(\|\widetilde{\bmu}(z,\cdot)\|_{L^2(\mathcal K_{R_0})}+|c(z)|\big)\\
			&\leq C\|\widetilde{\bmu}(z,\cdot)\|_{L^2(\mathcal K_{R_0})},
		\end{align*}
		for some $C>0$ independent of $z\in \overline{B}$, where we used \eqref{firstTrace} in the last step.
		Next, by combining \eqref{ATrace} with \eqref{gradientNorm} and the previous bound, we obtain that
		$$
		|\nabla\omega(z)|+|c(z)|\leq C \left(\int_{\mathcal K_{R_0}}|\widetilde{\bmu}(z,k)|^2\dA(k)\right)^{1/2}
		\quad\text{for almost every }z\in\C.
		$$
		To finish the proof, let $r$ be the radius of the disk $B$. For $z\in B$ and $k\in\mathcal K_{R_0}$,
		$$
		|e^{-\i kz}|\leq e^{2rR_0}.
		$$
		The relation $\widetilde{\bmu}=e^{-\i kz}\widetilde{\bpsi}$ therefore gives 
		$$
		\int_{\mathcal K_{R_0}}|\widetilde{\bmu}(z,k)|^2\dA(k)\leq e^{4rR_0}\, \int_{\mathcal K_{R_0}}|\widetilde{\bpsi}(z,k)|^2\dA(k),
		$$
		on $B$, which proves \eqref{fundamental_psi} there. Propositions~\ref{prop_exterior_primitives} and~\ref{prop_paired_diff} show that all terms in these inequalities have compact support in $B$. Thus the inequality holds on the whole plane. 
	\end{proof}

	\section{Hilbert-valued unique continuation}\label{sec_carleman}
	
	The final step of the proof requires a weak unique continuation
	principle for functions taking values in a Hilbert space. The
	corresponding scalar result for Schr\"odinger inequalities in the
	plane with $L^p$ coefficients, $p>1$, is classical; see
	\cite[Theorem 2]{ABG81}. We also refer to the foundational work
	of Jerison and Kenig \cite{JK85} on strong unique continuation for
	Schr\"odinger operators with critical potentials in higher dimensions.
	Since we need the planar result at the precise Sobolev regularity
	arising below, with constants independent of the target Hilbert space,
	we include the short proof.
	
	Let $\mathcal H$ be a separable complex Hilbert space. All
	Sobolev spaces below are Bochner spaces, and
	distributional derivatives are understood in the $\mathcal H$-valued
	sense. If $O\subset\C$ is open and $v,F\in L^1_{\loc}(O;\mathcal H)$, then
	$\Delta v=F$ in $\mathcal D'(O;\mathcal H)$ means that
	$$
	\int_O v(z)\Delta\varphi(z)\,\dA(z)
	=
	\int_O F(z)\varphi(z)\,\dA(z)
	\quad\text{in $\mathcal H$}
	$$
	for every scalar test function $\varphi\in C^\infty_c(O)$.
	
	\begin{lemma}\label{lem_carleman}
		Let $p\in (1,2)$, $a\in\C$, $\rho>0$, and $n\in\mathbb N$. Let $B(a,\rho)$ denote the open disk of radius $\rho$ centred at $a$. If
		$$
		v\in W^{1,\pt}(\C;\mathcal H),\qquad
		\Delta v\in L^p(\C;\mathcal H),\qquad
		\supp v\subset B(a,\rho)\setminus\{a\},
		$$
		then
		\begin{equation}\label{spatial_Carleman}
			\big\||z-a|^{-n}v\big\|_{L^\infty(\C;\mathcal H)}
			\leq C_p\,\rho^{2-\frac{2}{p}}\,
			\big\||z-a|^{-n}\Delta v\big\|_{L^p(\C;\mathcal H)},
		\end{equation}
		where $C_p>1$ depends only on $p$ and is in particular independent of $a$, $\rho$, $n$, and $\mathcal H$.
	\end{lemma}
	
	\begin{proof}
		We start with some observations about Cauchy transforms of Hilbert-valued functions where the Cauchy transforms act by Bochner integration. Firstly, for almost every $z\in\C$,
		$$
		\big|\bpartial^{-1}f(z)\big|_{\mathcal H}
		\leq
		\frac1\pi\int_{\C}
		\frac{|f(w)|_{\mathcal H}}{|z-w|}\,\dA(w).
		$$
		Applying the Hardy-Littlewood-Sobolev inequality
		to the function $|f|_{\mathcal H}$, we obtain
		\begin{equation}\label{hilbertCauchy}
			\|\bpartial^{-1}f\|_{L^{\pt}(\C;\mathcal H)}
			\leq C\|f\|_{L^p(\C;\mathcal H)}.
		\end{equation}
		Secondly, we also observe that for $g\in L^{\pt}(\C;\mathcal H)$ supported in $\overline{B(a,\rho)}$, H\"older's inequality gives
		\begin{equation}\label{hilbertCauchyInfinity}
			\|\partial^{-1}g\|_{L^\infty(\C;\mathcal H)}
			\leq C\,\rho^{2-\frac{2}{p}}\,\|g\|_{L^{\pt}(\C;\mathcal H)},
		\end{equation}
		since $\pt':=\pt/(\pt-1)<2$ and
		$$
		\sup_{z\in\C}\int_{B(a,\rho)}|z-w|^{-\pt'}\dA(w)
		\leq C\rho^{2-\pt'}.
		$$
		
		To prove the lemma, first suppose that $v$ is smooth and $\supp v \subset B(a,\rho) \setminus \{a\}$. The Cauchy representation gives
		\begin{align*}
			(z-a)^{-n}\partial v
			&=\frac14\bpartial^{-1}\big((z-a)^{-n}\Delta v\big),\\
			(\overline z-\overline a)^{-n}v
			&=\partial^{-1}\big((\overline z-\overline a)^{-n}\partial v\big).
		\end{align*}
		Apply \eqref{hilbertCauchy} to the first identity and \eqref{hilbertCauchyInfinity} to the second. Since the two powers have the same modulus, this proves \eqref{spatial_Carleman}. 
		
		For the general case, let $v_\varepsilon=\eta_\varepsilon*v$,
		where $\eta_\varepsilon$ is a standard scalar mollifier. Since
		$\supp v\subset B(a,\rho)\setminus\{a\}$, the functions
		$v_\varepsilon$ are supported in a fixed compact subset of
		$B(a,\rho)\setminus\{a\}$ for all sufficiently small $\varepsilon$.
		Moreover,
		$$
		v_\varepsilon\longrightarrow v
		\quad\text{in }W^{1,\pt}(\C;\mathcal H),
		\qquad
		\Delta v_\varepsilon
		=(\Delta v)*\eta_\varepsilon
		\longrightarrow\Delta v
		\quad\text{in }L^p(\C;\mathcal H).
		$$
		The Sobolev embedding
		$W^{1,\pt}(\C;\mathcal H)\hookrightarrow L^\infty(\C;\mathcal H)$,
		together with the fact that the weights are bounded on this fixed
		compact set, allows us to pass to the limit.
	\end{proof}
	
	\begin{proposition}\label{prop_weak_con}
		Let $p\in (1,2)$ and let $O\subset\C$ be a nonempty open and connected set. Suppose that
		$$
		u\in W^{1,\pt}(O;\mathcal H),\qquad
		\Delta u\in L^p(O;\mathcal H),
		$$
		and also that
		\begin{equation}\label{hilbertDifferentialInequality}
			|\Delta u|_{\mathcal H}\leq V|u|_{\mathcal H}
			\quad\text{almost everywhere in $O$},
		\end{equation}
		for some nonnegative function $V\in L^p(O)$. If $u$ vanishes on a nonempty open subset of $O$, then $u=0$ on $O$.
	\end{proposition}
	
	\begin{proof}
		Since $p \in (1,2)$ and $V\in L^p(O)$, we may fix $\rho_0>0$ such that
		\begin{equation}\label{local_abs}
			C_p\,\rho_0^{2-\frac{2}{p}}\,\|V\|_{L^p(O)}\leq\frac{1}{2},
		\end{equation}
		where $C_p$ is the constant in Lemma~\ref{lem_carleman}. Let $0<\rho\leq\rho_0$ be such that $\overline{B(a,\rho)}\subset O$. Suppose that $u$ vanishes near $a$. We first show that $u$ vanishes on $B(a,\rho/4)$.
		
		Take $\chi\in C^\infty_c(B(a,\rho))$, with $0\leq\chi\leq1$ and $\chi=1$ on $B(a,\rho/2)$, and define $v=\chi u$ on $O$ and extend it by zero to $\C$. Its distributional Laplacian is
		$$
		\Delta v=\chi\Delta u+2\nabla\chi\cdot\nabla u+(\Delta\chi)u.
		$$
		Thus $v$ satisfies the hypotheses of Lemma~\ref{lem_carleman} and
		$$
		|\Delta v|_{\mathcal H}\leq V|v|_{\mathcal H}+h,\qquad
		h=2|\nabla\chi||\nabla u|_{\mathcal H}+|\Delta\chi||u|_{\mathcal H}.
		$$
		Here $h\in L^p$ has support where $\rho/2\leq|z-a|<\rho$; its $L^p$-integrability follows from $\pt>p$ and the compact support of the derivatives of $\chi$. Writing $w_n=|z-a|^{-n}$, the lemma gives
		$$
		\|w_nv\|_{L^\infty}
		\leq C_p\rho^{2-\frac{2}{p}}
		\left(\|V\|_{L^p(O)}\|w_nv\|_{L^\infty}
		+(\rho/2)^{-n}\|h\|_{L^p}\right).
		$$
		Since $\rho\leq\rho_0$, \eqref{local_abs} allows us to absorb the first term on the right. Consequently,
		$$
		\|u\|_{L^\infty(B(a,\rho/4);\mathcal H)}
		\leq 2C_p\,\rho^{2-\frac{2}{p}}\, 2^{-n}\|h\|_{L^p}.
		$$
		All quantities on the right except $2^{-n}$ are fixed, so $n\to\infty$ proves the local assertion.
		
		Let $Z$ consist of the points having a neighborhood on which $u=0$. It is nonempty and open. If $x\in\overline Z\cap O$, choose $0<\rho\leq\rho_0$ with $\overline{B(x,2\rho)}\subset O$, and take $a\in Z$ such that $|a-x|<\rho/8$. Then $\overline{B(a,\rho)}\subset O$, so the local assertion gives $u=0$ on $B(a,\rho/4)$, hence near $x$. Thus $Z$ is relatively closed, and connectedness gives $Z=O$.
	\end{proof}
	
	\section{Proof of Theorem~\ref{main_thm}}\label{sec_uniqueness}
	
	\begin{proof}[Proof of Theorem~\ref{main_thm}]
		We assume without loss of generality that $1<p<2$. Fix $R_0>k_0$ so that \eqref{chooseR} is satisfied, and put
		\begin{equation}\label{frequencyHilbertSpace}
			\mathcal H=L^2(\mathcal K_{R_0};\C^2),\qquad
			U(z)=\mathcal W_1(z)^{-1}\widetilde{\bpsi}(z,\cdot),
		\end{equation}
		where $\mathcal K_{R_0}$ is as in \eqref{frequency_ann} and we recall the notations in \eqref{Epsi}-\eqref{difference_def}. We identify each family $\bpsi_j$ with the $\mathcal H$-valued map $z\mapsto\bpsi_j(z,\cdot)$. Throughout the proof, we work with the restrictions
		of these maps and $U$ to $B$, retaining the same notation. All distributional identities below are understood in the spatial variable on $B$.
		
		Writing $U(z)(k)=(U_1(z,k),U_2(z,k))^{\mathsf T}$, we have
		\begin{equation}\label{spatial_fam}
			U_1=\psi_1-e^\omega\psi_2,\qquad
			U_2=\psi_1^c-e^{\overline\omega}\psi_2^c,
		\end{equation}
		where we recall the notation \eqref{mu_j}. Both $\mathcal W_1$ and its inverse are bounded on $B$. Hence Proposition~\ref{prop_fundamental} gives a $C\geq 1$ such that
		\begin{equation}\label{gradient_fam}
			|\nabla\omega(z)|\leq C|U(z)|_{\mathcal H}
			\quad\text{for almost every $z\in B$}.
		\end{equation}
		Proposition~\ref{prop_paired_diff} also shows that $U$ vanishes on a neighborhood of $\partial B$.
		
		We verify the hypotheses of Proposition~\ref{prop_weak_con}.
		Lemma~\ref{lem_cgo}, together with the boundedness of $B$ and
		$\mathcal K_{R_0}$, gives
		\begin{equation}\label{uniform_spatial_fam}
			\sup_{k\in\mathcal K_{R_0}}\sum_{j=1}^2
			\left(
			\|\bpsi_j(\cdot,k)\|_{L^\infty(B)}
			+
			\|\bpsi_j(\cdot,k)\|_{W^{1,\pt}(B)}
			\right)<\infty.
		\end{equation}
		By Lemma~\ref{lem_kregularity}, these families are strongly
		measurable. Since $\mathcal K_{R_0}$ has finite measure and
		$\pt>2$, Minkowski's integral inequality and Fubini's theorem give
		$$
		\bpsi_j\in
		W^{1,\pt}(B;\mathcal H)\cap L^\infty(B;\mathcal H),
		\qquad j=1,2.
		$$
		Since
		$$
		\omega\in W^{1,\pt}(B)\cap L^\infty(B),
		$$
		it follows that
		$$
		U\in W^{1,\pt}(B;\mathcal H)\cap L^\infty(B;\mathcal H).
		$$
		Define
		$$
		\mathcal G_1(z):=|\bpsi_2(z,\cdot)|_{\mathcal H},
		\qquad
		\mathcal G_2(z):=|\nabla\bpsi_2(z,\cdot)|_{\mathcal H}.
		$$
		Then,
		$$
		\mathcal G_1\in L^\infty(B),
		\qquad
		\mathcal G_2\in L^{\pt}(B).
		$$
		
		Recall also that for every $k\in\mathcal K_{R_0}$,
		$$
		\psi_j(\cdot,k),\ \psi_j^c(\cdot,k),\ \omega
		\in W^{2,p}(B).
		$$
		Using $\Delta\omega=q_1-q_2$, we obtain
		\begin{equation}\label{spatial_fam_eq}
			(-\Delta+q_1)U_1
			=
			e^\omega\left(
			2\nabla\omega\cdot\nabla\psi_2
			+(\nabla\omega\cdot\nabla\omega)\psi_2
			\right),
		\end{equation}
		where we recall the dot products above are complex bilinear. The second component satisfies the analogous identity obtained by
		replacing $(q_1,\omega,\psi_2)$ with
		$(\overline{q_1},\overline\omega,\psi_2^c)$.
		
		Let $F(z,k)$ denote the vector formed by the two right hand sides.
		Then
		$$
		|F(z,\cdot)|_{\mathcal H}
		\leq
		C\left(
		|\nabla\omega(z)|\,\mathcal G_2(z)
		+
		|\nabla\omega(z)|^2\,\mathcal G_1(z)
		\right).
		$$
		Since $\nabla\omega,\mathcal G_2\in L^{\pt}(B)$ and
		$\mathcal G_1\in L^\infty(B)$, the right hand side belongs to
		$L^{\pt/2}(B) \subset L^{p}(B)$ implying that 
		$$
		F\in L^p(B;\mathcal H).
		$$
		Fubini's theorem applied to the componentwise distributional
		identities now gives
		$$
		\Delta_z U=\mathsf Q_1U-F
		\quad\text{in }\mathcal D'(B;\mathcal H),
		\qquad
		\mathsf Q_1=\diag(q_1,\overline{q_1}).
		$$
		In particular, $\Delta_zU\in L^p(B;\mathcal H)$, and
		$$
		|\Delta_zU|_{\mathcal H}
		\leq
		|q_1|\,|U|_{\mathcal H}
		+
		C\left(
		\mathcal G_2+|\nabla\omega|\mathcal G_1
		\right)|\nabla\omega|.
		$$
		Combining this with \eqref{gradient_fam}, we obtain
		$$
		|\Delta_zU|_{\mathcal H}
		\leq V|U|_{\mathcal H},
		$$
		where
		$$
		V=
		|q_1|
		+
		C\left(
		\mathcal G_2+|\nabla\omega|\mathcal G_1
		\right)
		\in L^p(B).
		$$

		Proposition~\ref{prop_weak_con}, applied with $O=B$,
		therefore gives $U=0$ on $B$.
		Equation~\eqref{gradient_fam} then yields
		$\nabla\omega=0$ almost everywhere in $B$.
		Since $\Delta\omega=q_1-q_2$, we conclude that
		$q_1=q_2$ almost everywhere in $\Omega$.
	\end{proof}

\end{document}